\documentclass[11pt]{article}
\usepackage{amsmath,amsthm,amsfonts,amssymb,amscd, amsxtra, mathrsfs}
\usepackage{url}
\usepackage[margin=2.5cm,nohead]{geometry}
\usepackage{color}
\newcommand{\banacha}{X}
\newcommand{\banachb}{Y}
\newcommand{\banachc}{Z}
\newtheorem{theorem}{Theorem}
\newtheorem{lemma}[theorem]{Lemma}
\newtheorem{definition}{Definition}
\newtheorem{corollary}[theorem]{Corollary}
\newtheorem{proposition}[theorem]{Proposition}

\newtheorem{remark}{Remark}
\begin{document}
\title{Local convergence analysis  of Newton's  method for solving strongly regular generalized equations}

\author{ 
O.  P. Ferreira\thanks{IME/UFG,  CP-131, CEP 74001-970 - Goi\^ania, GO, Brazil (Email: {\tt
      orizon@ufg.br}).  The author was supported in part by
FAPEG, CNPq Grants 305158/2014-7 and PRONEX--Optimization(FAPERJ/CNPq).} 
\and 
G. N. Silva \thanks{CCET/UFOB,  CEP 47808-021 - Barreiras, BA, Brazil (Email: {\tt  gilson.silva@ufob.edu.br}). The author was supported in part by CAPES .}    
}
\date{September  26, 2016}

\maketitle

\begin{abstract}
In this paper we study Newton's method for solving generalized equations in Banach spaces. We show that under strong regularity of the generalized equation, the  method is locally  convergent  to a solution with  superlinear/quadratic  rate.  The presented analysis  is  based on Banach Perturbation Lemma for generalized equation and   the classical Lipschitz condition on the derivative is relaxed  by  using a general  majorant function, which enables  obtaining   the optimal convergence radius,  uniqueness of  solution   as well as unifies  earlier results pertaining to  Newton's method theory. 
\medskip
 
\end{abstract}

\section{Introduction}
\label{intro}
In this paper we consider the Newton's method for solving the generalized equation of the form
\begin{equation} \label{eq:ipi}
  f(x) +F(x) \ni 0,
\end{equation}
where $f:{\Omega}\to \banachb$ is a   continuously   differentiable function, $\banacha$  and $\banachb$ are Banach spaces, $\Omega\subseteq \banacha$ is  an open set and $F:\banacha \rightrightarrows  \banachb$ is a set-valued mapping with  nonempty  closed graph. As is well known, the model of the generalized equation \eqref{eq:ipi} covers several class of problems, due to this  important characteristic  it has been studied  in several works,  having \cite{DontchevAragon2014, DontchevAragon2011, Dontchev1996, DontchevRockafellar2009,DontchevRockafellar2010, DontchevRockafellar2013,   josephy1979, Robinson1972_2} as part of a whole.  For instance,  as we can see, if  $F$ is the normal cone mapping $N_C$, of a convex set $C$ in $\banachb$ and $\banachb=\banacha^*$ is the dual of $\banacha$,  the inclusion \eqref{eq:ipi} is the  variational inequality problem; for more details see  \cite{DontchevRockafellar2009}.

Newton's method is undoubtedly one of the most popular methods for numerically solving nonlinear  equation. This is because of its importance both theoretical and practical, and even more is due to its quadratic rate of convergence. Throughout the years, this method has been extended in many directions by several authors,   one of the most studied currently is the generalization of this  to solve \eqref{eq:ipi}, which has its origin in the works of N. H.  Josephy \cite{josephy1979}. Based on the work of  N. H.  Josephy \cite{josephy1979}, we study the local convergence of the following Newton's method for solving \eqref{eq:ipi}:
\begin{equation} \label{eq:ipi1}
  f(x_k) + f'(x_k)(x_{k+1}-x_k)+ F(x_{k+1}) \ni 0, \qquad k=0,1, \ldots.
\end{equation}
This algorithm has been studied in several papers  including  but not limited to \cite{DontchevAragon2014, DontchevAragon2011, Bonnans, DokovDontchev1998, Dontchev1996, Dontchev1996_2,  DontchevRockafellar2013};  see also \cite[ Section~6C]{DontchevRockafellar2009} and \cite{LiNg}. If $F\equiv 0$, the iteration \eqref{eq:ipi1} becomes the usual Newton method for solving the equation $f(x)=0.$ If  $F= N_C$,  the normal cone mapping   of a convex set $C$ in $\banachb$ and $\banachb=\banacha^*$,  then   \eqref{eq:ipi1} is the version of the Newton's method for solving  variational inequality;  see \cite{Bonnans, DokovDontchev1998, josephy1979}. In particular, if \eqref{eq:ipi} represents the Karush-Kuhn-Tucker  optimality conditions  for a mathematical programming problem, then  \eqref{eq:ipi1} describes the well-known   sequential quadratic programming method; see for example a detailed discussion in   \cite[pag. 334]{DontchevRockafellar2009}; see also \cite{Dontchev1996_3}.

Under the assumption that $f$ is Fr\'echet differentiable in some neighborhood of a solution $\bar x$ of \eqref{eq:ipi},  S.~M.~Robinson in \cite{Robinson1980} obtained a condition on the linearization of \eqref{eq:ipi} about $\bar{x}$, i.e., on the generalized equation
$$
  f(\bar{x}) + f'(\bar{x})(x-\bar{x})+ F(x) \ni 0 ,
$$
which he called {\it strong regularity}, in order to guarantee unique solution of the generalized equation $$f(\hat{x}) + f'(\hat{x})(x-\hat{x})+ F(x) \ni 0 ,$$ for all $\hat{x}$ in  a neighborhood  $\bar x$. The classic local analysis of Newton's method  for solving  $f(x)=0$ require invertibility  of the derivative $f'$ at the solution,  which is actually critical for  the  well definition of the method.  Therefore, for the  local analysis of Newton's method  for solving  \eqref{eq:ipi} we will need of a similar  concept, namely,  the {\it strong regularity} of $f+F$ at the solution $\bar{x}\in  \banacha$  for $0\in \banachb$. If $\banacha = \banachb$ and $F=\{ 0\},$ then   strong regularity of $f+F$ at the solution $\bar{x}\in \banacha$  for $0\in  \banacha$ is equivalent to assumption that   $f'(\bar{x})^{-1}$  is a continuous linear operator. An important case is when \eqref{eq:ipi} represents the  Karush-Kuhn-Tucker's  system for  the standard nonlinear programming problem with a strict local minimizer, see  \cite{DontchevRockafellar2009} pag. 232. In this case, strong regularity of this system is equivalent to  the  linear independence of the gradients of the active constraints and a strong form of the second-order sufficient optimality condition; for details see  \cite[Theorem 6]{DontchevRockafellar96}.  The analysis presented in this paper will be made under  strong regularity on the solution of \eqref{eq:ipi}. 

It is well-known  that, to obtain quadratic convergence rate of Newton's method \eqref{eq:ipi1}, the Lipschitz continuity of  $f'$  in a neighborhood  of a solution of \eqref{eq:ipi} is required, see for example  \cite{  DontchevAragon2011, Dontchev1996, DontchevRockafellar2009, DontchevRockafellar2010}.  Indeed,   keeping control of the derivative is an important point in the convergence analysis of Newton's methods and its variations, as we can see in \cite{ Dontchev2015, DontchevRockafellar2013,  DontchevAragon2014,  FerreiraMax2013, Wang2002206}. Recently, there has been an increased interest in the study of Newton's method and its variations for solving the equation $f(x)=0$,   by relaxing the hypothesis of Lipschitz continuity of $f'$. For instance, the majorant condition is one of those conditions that relax the Lipschitz condition which has several equivalent formulations,  see for example \cite{Alvarez2008, Argyros, FerreiraSvaiter2009, FerreiraMax2013,  LiWang2006, LiWangDedieu2009,  Wang1999, Wang2002206, Zabrejko1987}.  The advantage of working with a majorant condition is that it makes us clearly see how big the radius of convergence is, besides    allow us  unify several convergence results pertaining to  Newton's method; see  \cite{FerreiraSvaiter2009, Wang1999}. In this paper, under the majorant condition,  we establish a local convergence analysis of Newton's method \eqref{eq:ipi1} by assuming  strong regularity of $f+F$ at the solution $\bar{x}\in \banacha$  for $0\in  \banacha$. Before proving our main result, which establish the optimal convergence radius for  the method with respect to the majorant condition and uniqueness of solution, a clear relationship between the majorant function and the function defining the generalized equation is obtained. As special cases, we present an analysis of this result under Lipschitz's  and Smale's conditions. Up to our knowledge, this is the first time that the Newton method for solving  generalized equations under a general majorant condition and, in particular, under Smale's condition is analyzed,  similar studies  has been done in \cite{Alvarez2008, DedieuPrioureMalajovich2003, Gilson2016,  Gilson2016_1,  LiWangDedieu2009,  Li2008a, LiWang2006, Wang2002206}.  In addition, it is worth mentioning that  our  approach is based in the  Banach Perturbation Lemma obtained by  S.~M.~Robinson in  \cite[Theorem~2.4]{Robinson1980}. In this sense,  our approach is related to the techniques  used in  \cite{Dontchev2015, DokovDontchev1998, josephy1979}.

The organization of the paper is as follows. In the following section we present background material and some technical results  used  throughout  the paper. Section \ref{lkant} is devoted to our main result and  in   Section~\ref{sec:PMF} properties of the majorant function,  the main  relationships   between the majorant function and the generalized equation, the uniqueness of the solution and the optimal convergence radius are established. In  Section~\ref{sec:proof} the main result is then proved and some applications of this result are given in Section~\ref{apl}. Finally, conclusions are stated in  Section~\ref{rf}.

\section{Preliminaries} \label{sec:int.1}
We use the following notation. Let $\banacha$,\, $\banachb$ be Banach spaces.  The {\it space consisting of all continuous linear mappings} $A:\banacha \to \banachb$ will be denote by ${\mathscr L}(\banacha,\banachb)$  and   the {\it operator norm}  of $A$ will be defined  by $  \|A\|:=\sup \; \{ \|A x\|~:  \|x\| \leqslant 1 \}.$ Let  $\Omega\subseteq \banacha$ and    $h:{\Omega}\to \banachb$  a function with  Fr\'echet derivative at  all $x\in int(\Omega)$.  The Fr\'echet derivative of $h$ at $x$ is the linear map $h'(x):\banacha \to \banachb$ which is continuous. We identify as the {\it graph} of the set-valued mapping $H:\banacha \rightrightarrows  \banachb$ the set $\mbox{gph}~H:= \left\{(x,y)\in \banacha \times \banachb ~: ~ y \in H(x)\right\}$.  The {\it inverse} $H^{-1}$ of a map $H:\banacha \rightrightarrows  \banachb$  is defined as  $ H^{-1}(y):=\{x \in \banacha ~: ~ y \in H(x)\}$. Define  $ B(x,\delta) := \{ y\in X ~: ~\|x-y\|<\delta \}$ and  $B[x,\delta] := \{ y\in X ~: ~\|x-y\|\leqslant \delta\}$ as the {\it open} and {\it closed balls} centered at $x$ with radius $\delta\geq 0$.
\begin{definition}\label{def:plm}
Let  $\Omega$ be a nonempty, open, convex subset of $\banacha$. Let  $h: \Omega \to \banachb$ be   a function  with continuous derivative  $h'$ and $H:\banacha \rightrightarrows  \banachb$ be a set-valued mapping. The {\it partial linearization mapping} of    $h +H$  at   $x$  is   the set-valued mapping   $L_h(x, \cdot ):\banacha \rightrightarrows  \banachb$  given by
$$
L_h(x, z ):=h(x)+h'(x)(z-x)+H(z).
$$
Thus, the linearization of  a generalized equation $h(z)+H(z)\ni 0$ at $x$ is obtained by replacing $h+H$ with $L_h(x, \cdot)$. The inverse $L_h(x, \cdot )^{-1}$ of the map  $L_h(x, \cdot )$  at $y\in Y$ is denoted  by
\begin{equation} \label{eq:invplm}
L_h(x, y )^{-1}:=\left\{z\in X ~:~ y\in h(x)+h'(x)(z-x)+H(z)\right\}.
\end{equation}
\end{definition}
An important element in the analysis of Newton's method,  for solving the equation $f(x)=0,$ is the behavior of the inverse $f'(x)^{-1}$,  for $x$ in a neighborhood of a solution $\bar{x}.$ The analogous element for  the generalized equation \eqref{eq:ipi} is the behavior  of the inverse mapping  $L_f(x, \cdot )^{-1}$,  for $x$ in a neighborhood of a solution $\bar{x}.$ It is worth to point out that,  N.~H.~Josephy  in \cite{josephy1979} was the first to consider  Newton's method for solving the generalized equation $f(x)+N_{C}(x) \ni 0$, where $N_{C}$ is the normal cone of a convex set $C\subset \mathbb{R}^n$, by defining the Newton iteration as   
$$
f(x_k)+f'(x_k)(x_{k+1} -x_k)+N_{C}(x_{k+1} ),  \qquad k=0, 1, \ldots. 
$$
For analyzing  this method,   was  employed the  important concept  of strong regularity defined by S.~M.~Robinson \cite{Robinson1980}, which  assure  {\it ``good  behavior"} of $L_f(x, \cdot )^{-1}$,  for $x$ in a neighborhood of a solution $\bar{x}.$ Here we adopt the following definition   due to S.~M.~Robinson;   see \cite{Robinson1980}.
\begin{definition}\label{eq:stronmetr}
Let $\banacha$,\, $\banachb$ be Banach spaces,    $\Omega$ be an open and nonempty subset of $\banacha$,  $h: \Omega \to \banachb$ be   Fr\'echet differentiable  with derivative  $h'$ and $H:\banacha \rightrightarrows  \banachb$ be a set-valued mapping.  The mapping $h+H$ is said to be strongly regular at $x$ for $y$ with modulus $\lambda>0$,  when $y \in h(x) +H(x)$ and there exist neighborhoods $U_{x}$ of $x$ and $V_{y}$ of $y$ in $\banachb$  such that  $U_{x}\subset \Omega$,  the mapping $ z\mapsto L_h(x, z )^{-1}\cap U_{x}$ is a  single-valued   function from $ V_{y}$ to  $ U_{x}$,  which is Lipschitizian on $V_y$  with modulus  $\lambda$, i.e.,
$$
\left\|L_h(x,u)^{-1}\cap U_{x}- L_h(x,v)^{-1}\cap U_{x} \right\| \leq \lambda \|u-v\|, \qquad \forall ~ u, v \in  V_{y}.
$$
\end{definition}
Since the mapping $ z\mapsto L_h(x,z)^{-1}\cap U_{x}$ is   single-valued  from $ V_{y}$ to  $ U_{x}$, for simplify the notation we are using in above definition  $w= L_h(x,z)^{-1}\cap U_{x}$ instead of $\{w\}:= L_f(x,z)^{-1}\cap U_{x}.$ {\it From now on we will use  this simplified notation. }
\begin{remark} \label{re:rxinf}
If in above definition $H(x)\equiv \{0\}$ then  the property of $h+H\equiv h$ be strongly  regular  at the solution $\bar{x}$  for $0$,  reduces to $h'(\bar{x})$ has an  inverse $h'(\bar{x})^{-1}$. Moreover,  in this case,  $\lambda=\|h'(\bar{x})^{-1}\|$, $U_{\bar{x}}=X$ and $V_{\bar{x}}= \banachb$.  An important particular instance is when \eqref{eq:ipi} represents the  Karush-Kuhn-Tucker's  system for  the standard nonlinear programming problem with a strict local minimizer.  In this case,  the strong  regularity of this system is equivalent to  the  linear independence of the gradients of the active constraints and the strong second-order sufficient optimality condition;   see  \cite[Example 6C.8]{DontchevRockafellar2009}, see also  \cite[Theorem~6]{DontchevRockafellar96}.
\end{remark}
For a detailed discussion about  Definition~\ref{eq:stronmetr}  see \cite{DontchevRockafellar2009, Robinson1980}.  The next result is a type of implicit function theorem for generalized equations satisfying the condition of strong regularity, its proof is similar to  \cite[Theorem 2.1] {Robinson1980}, it  also can be seem as a particular instance of   \cite[Theorem 5F.4]{DontchevRockafellar2009} on page 294.
\begin{theorem}\label{eq:Implitheor}
 Let $\banacha$,  $\banachb$ and  $\banachc$ be Banach spaces, $G:\banacha \rightrightarrows  \banachb$ be  a set-valued mapping and $g:\banachc \times\banacha \to \banachb$  be a  continuous   function,   having  partial Fr\'echet derivative  with respect to the second variable  $D_{x}g$  on $\banachc\times\banacha$, which is also continuous. Let $\bar{p} \in \banachc$ and suppose  that  $\bar x$ solves the generalized equation
$$
g(\bar p,x)+G(x) \ni 0.
$$
Assume that the  mapping  $g(\bar{p}, .)+G$ is strongly regular at $\bar{x}$ for $0$, with modulus $\lambda$. Then, for any    $\epsilon >0$  there exist neighborhoods $U_\epsilon$ of $\bar x$ and $V_\epsilon$ of $\bar p$  and a single-valued function $s:V_\epsilon \to U_\epsilon$ such that  for any $p\in V_\epsilon$, $s(p)$ is the unique solution in $U_\epsilon$ of the inclusion
$
g(p,x)+G(x) \ni 0,
$
and $s(\bar p)=\bar x$. Moreover, there holds
$$
\|s(p')-s(p)\| \leq (\lambda + \epsilon)  \|g(p',s(p))- g(p,s(p))\|,  \qquad \forall~ p,p' \in V_\epsilon.
$$
\end{theorem}

 Indeed, the  first version of the Theorem~\ref{eq:Implitheor} was proved by S.~M.~Robinson; see  \cite[Theorem~{2.1}]{Robinson1980},   to the particular  case  $F=N_{C}$, where $N_{C}$ is the normal cone of a convex set $C\subset \banacha$. As an application,  a version of the Banach Perturbation Lemma    involving the normal cone was  obtained;  see \cite[Theorem~2.4]{Robinson1980}. N.~H.~Josephy in \cite{josephy1979},  used  this version of  Banach Perturbation Lemma,  see \cite[Corollary~1]{josephy1979},  for  proving  that the Newton iteration
$$
  f(x_k) + f'(x_k)(x_{k+1}-x_k)+ N_{C}(x_{k+1}) \ni 0, \qquad  k=0,1,\ldots,
$$
where $N_{C}$ is the normal cone of a convex set $C\subset \mathbb{R}^n$, is quadratically convergent to a solution of   $f(x)+N_{C}(x) \ni 0$.  In the next lemma we state a   version  of the Banach Perturbation Lemma  involving a general set-valued mapping, its proof is similar to the correspondent   one \cite[Theorem~2.4]{Robinson1980}.
\begin{lemma} \label{lem:blr}
Let $\banacha, \banachb$ be Banach spaces, $a_0$ be a point of $\banachb,$ $G:\banacha \rightrightarrows  \banachb$ be a set-valued mapping and $A_0:\banacha \to \banachb$ be a bounded linear operator. Suppose that $\bar{x}$ is a point of $\banacha$ which satisfies the generalized equation
$$
0 \in A_0 x + a_0 + G(x).
$$
Assume that the mapping $ A_0  + a_0 + G$ is   strongly  regular at $\bar{x}$ for $0$ with modulus $\lambda$. Then there exist neighborhoods $M$ of $A_0$ in ${\mathscr L}(\banacha,\banachb),$ $N$ of $a_0$ and $W$ of origin in $\banachb,$ and $U$ of $\bar{x},$ such that,  for any $A\in M$ and $a\in N$,   letting   $T(A, a, \cdot): U \rightrightarrows  \banachb$ be   defined as
$$
T(A, a, x):= Ax + a + G(x),
$$
then $ y\mapsto T(A, a, y)^{-1}\cap U$ is a  single-valued  function from $W$ to $U$.  Moreover, letting a neighborhood $\bar{M}$ of $A_0$ such that   $\bar{M}\subset M$ and $\lambda \|A-A_0\|<1$   for each $A\in \bar{M}$, then for each $A\in \bar{M}$ and $a\in N$ the function $ y\mapsto T(A, a, y)^{-1}\cap U$ is    Lipschitz   on  W as follows
$$
\left\|T(A, a, y_1)^{-1}\cap U - T(A, a, y_2)^{-1}\cap U \right\| \leq \frac{\lambda}{1-\lambda\|A-A_0\|} \|y_1-y_2\|,  
$$
for each $y_1, y_2 \in  W.$
\end{lemma}
Next we  establish a corollary to Lemma~\ref{lem:blr}, which  will have important role in the sequel.  A similar result has been obtained by S.~P.~Dokov and  A.~L.~Dontchev, see lemma on pag. 119 of \cite{DokovDontchev1998},  for studying  the local quadratic convergence of Newton's method for variational inequality.
\begin{corollary} \label{cor:ban}
Let $\banacha$,\, $\banachb$ be Banach spaces,    $\Omega$ be an open nonempty subset of $\banacha$,  $f: \Omega \to \banachb$ be continuous with   Fr\'echet differentiable   $f'$ continuous, and $F:\banacha \rightrightarrows  \banachb$ be a set-valued mapping. Suppose that $\bar{x}\in \Omega $   and $f+F$ is  strongly regular at $\bar{x}$ for $0$ with modulus $\lambda >0$.   Then, there exist constants  $r_{\bar{x}}>0$ and $r_{0}>0$ such that, for any $x\in B(\bar{x}, r_{\bar{x}})$, there hold $\lambda \|f'(x)-f'(\bar{x})\|<1$,      the mapping $ z\mapsto L_f(x,z)^{-1}\cap B(\bar{x}, r_{\bar{x}})$ is   single-valued  from $B(0, r_{0})$ to  $B(\bar{x}, r_{\bar{x}})$  and  Lipschitizian on $B(0, r_{0})$  as follows
$$
\left\|L_f(x,u)^{-1}\cap B(\bar{x}, r_{\bar{x}})- L_f(x,v)^{-1}\cap B(\bar{x}, r_{\bar{x}}) \right\| \leq \frac{\lambda \|u-v\|}{1-\lambda\|f'(x)-f'(\bar{x})\|},
$$
$\forall ~ u, v \in  B(0, r_{0}).$
\end{corollary}
\begin{proof}
Since  $f+F$ is  strongly regular at $\bar{x}$ for $0$ with modulus $\lambda >0$,  thus  $\bar{x}$ is also solution of the generalized equation
 $$
0 \in L_{f}(\bar{x}, x)=f'(\bar{x}) x + f(\bar{x})-f'(\bar{x})\bar{x} +F(x),
$$
and the mapping $x\mapsto f'(\bar{x}) x + f(\bar{x})-f'(\bar{x})\bar{x} +F(x)$ from $\banacha$ to  $\banachb$ is strongly regular at $\bar{x}$ for $0$ with the same modulus $\lambda >0$. Thus, applying first part of Lemma~\ref{lem:blr} with    $A_0=f'(\bar{x}),$ $a_0= f(\bar{x})-f'(\bar{x})\bar{x}$ and  $G=F, $  we conclude  that   there exist neighborhoods $M$ of $f'(\bar{x})$ in ${\mathscr L}(\banacha,\banachb),$ $N$ of $f(\bar{x})-f'(\bar{x})\bar{x} $ in  $\banacha$ and $W=B(0, r_{0})$ of origin in $\banachb,$ and $U=B(\bar{x}, r_{\bar{x}})\subset \Omega$ of $\bar{x},$  where $r_{\bar{x}}>0$ and $r_{0}>0$, such that,  for any $A\in M$  and $a\in N$,   letting   $T(A, a, \cdot): B(\bar{x}, r_{\bar{x}}) \rightrightarrows  \banachb$ be   defined as
\begin{equation}\label{eq;tgco}
T(A, a, y):= Ay + a + F(y),
\end{equation}
the mapping $ z\mapsto T(A, a, z)^{-1}\cap B(\bar{x}, r_{\bar{x}})$ is a  single-valued function from $B(0, r_{0})$ to $B(\bar{x}, r_{\bar{x}})$. On the other hand,  due to  $f$  be   continuous with   Fr\'echet differentiable   $f'$ continuous on $\Omega$, we can  shrink  $r_{\bar{x}}$,  if necessary, such that
\begin{equation} \label{eq:cpmc}
\lambda \|f'(x)-f'(\bar{x})\|<1, \quad f'(x) \in M, \quad   f(x)-f'(x)x \in N, \quad \forall ~x\in B(\bar{x}, r_{\bar{x}}).
\end{equation}
Since Definition~\ref{def:plm} and \eqref{eq;tgco} imply that   $L_f(x, y)=T(f'(x), f(x)-f'(x)x, y)$, for all $y, x\in B(\bar{x}, r_{\bar{x}})$, after some manipulations we have, for each $z\in B(0, r_{0})$,
\begin{equation} \label{eq:elft}
L_f(x, z)^{-1}\cap B(\bar{x}, r_{\bar{x}})=T(f'(x), f(x)-f'(x)x, z)^{-1}\cap B(\bar{x}, r_{\bar{x}}), 
\end{equation}
for each $x\in B(\bar{x}, r_{\bar{x}})$. Therefore, for each $x\in B(\bar{x}, r_{\bar{x}})$,  the last equality and \eqref{eq:cpmc} imply that   $ z\mapsto L_f(x,z)^{-1}\cap B(\bar{x}, r_{\bar{x}})$ is   single-valued  from $B(0, r_{0})$ to  $B(\bar{x}, r_{\bar{x}})$, which  prove the first part of  corollary.   Finally, taking into account  \eqref{eq:elft} and   second  part of Lemma~\ref{lem:blr},  we conclude that the mapping   $ z\mapsto L_f(x,z)^{-1}\cap B(\bar{x}, r_{\bar{x}})$     is Lipschitzian    from $B(0, r_{0})$ to $B(\bar{x}, r_{\bar{x}})$ with constant    $\lambda/[1-\lambda\|f'(x)-f'(\bar{x})\|], $ which conclude the proof.
\end{proof}

%
%
\section{Local Convergence of the Newton method } \label{lkant}
In this section,  we state  our main result. We present an analysis of the behavior of the sequence  generated by Newton's method for solving the generalized equation \eqref{eq:ipi}.  For this purpose, we suppose that $f+F$ is strongly regular at $\bar{x}$ for $0$ with modulus $\lambda >0$, where $\bar x$ is such that $f(\bar x) +F(\bar x) \ni 0$. Moreover,  the Lipschitz continuity of  $f'$  is relaxed, i.e., we assume that $f'$ satisfies the  conditions of the next definition. 
\begin{definition}\label{def.majorcondition}
Let $\banacha$,\, $\banachb$ be Banach space,    $\Omega$ be an open nonempty subset of $\banacha$,    $f: \Omega \to \banachb$ be continuous with   Fr\'echet derivative   $f'$ continuous in $\Omega$. Let   $\bar{x}\in \Omega$ and   $R>0$  and  $\kappa:=\sup\{t\in [0, R): B(\bar{x}, t)\subset \Omega\}$. A  twice  continuously differentiable function  $\psi:[0,\; R)\to \mathbb{R}$ is a majorant function for $f$ on $B(\bar x, R)$ with  modulus $\lambda >0$ if it satisfies
\begin{equation}\label{Hyp:MH}
\lambda \left\|f'(x)-f'(\bar{x}+\tau(x-\bar{x}))\right\| \leq    \psi'\left(\|x-\bar{x}\|\right)-\psi'\left(\tau\|x-\bar{x}\|\right),
 \end{equation}
  for all $\tau \in [0,1]$, $x\in B(\bar{x}, \kappa)$ and, moreover,  there hold:
\begin{itemize}
  \item[{\bf h1)}]  $\psi(0)=0$ and $\psi'(0)=-1$;
  \item[{\bf  h2)}]  $\psi'$ is  strictly increasing.
\end{itemize}
\end{definition}
The statement of  the our main result is:
\begin{theorem}\label{th:nt}
Let $\banacha$,\, $\banachb$ be Banach spaces,    $\Omega$ be an open nonempty subset of $\banacha$,    $f: \Omega \to \banachb$ be continuous with   Fr\'echet derivative   $f'$ continuous in $\Omega$, $F:\banacha \rightrightarrows  \banachb$ be a set-valued mapping and $\bar{x}\in \Omega$. Suppose that  $f+F$ is  strongly regular at $\bar{x}$ for $0$ with modulus $\lambda >0$.   Let $R>0$  and  assume that $\psi:[0,\; R)\to \mathbb{R}$ is a majorant function for $f$ on $B(\bar x, R)$ with  modulus $\lambda >0$.   Let  $\nu:=\sup\{t\in [0, R): \psi'(t)< 0\},$ $\rho:=\sup\{t\in (0, \nu): \psi(t)/(t\psi'(t))-1<1\} $  and $r:=\min \left\{\kappa, \,\rho \right\}$.
Then, there exists a convergence radius  $r_{\bar{x}}>0$ with $r_{\bar{x}}\leq  r $ such that  the sequences  with starting point $x_0\in B(\bar{x},  r_{\bar{x}})/\{\bar{x}\}$ and $t_0=\|\bar{x}-x_0\|$, respectively, 
\begin{equation} \label{eq:DNS}
 0\in f(x_k)+f'(x_k)(x_{k+1}-x_k)+F(x_{k+1}), \quad t_{k+1} =|{t_k}-\psi(t_k)/\psi'(t_k)|, 
\end{equation}
$ k=0,1,\ldots\,,$ are well defined, $\{t_k\}$ is strictly decreasing, is contained in $(0, r)$ and converges to $0$, $\{x_k\}$ is contained in $B(\bar{x}, r_{\bar{x}})$ and  converges to the point $\bar{x}$ which is the unique  solution of the generalized equation $f(x)+F(x)\ni 0$ in $B(\bar{x}, \bar{\sigma})$, where $0<\bar{\sigma}\leq\min \{r_{\bar{x}}, \sigma\}$  and  $\sigma:=\sup\{0<t<\kappa: \psi(t)< 0\}$ and there hold
\begin{equation}\label{eq:linearconv}
\lim_{k\to \infty}\frac{\|x_{k+1}- \bar x\|}{\|x_k -\bar x\|}=0, \qquad \qquad \lim_{k \to \infty}\frac{t_{k+1}}{t_k}=0 .
\end{equation}
Moreover, given $0\leq p\leq 1$ and assume  that
\begin{itemize}
  \item[{\bf h3)}] the function $(0,\nu) \ni t\longmapsto [\psi(t)/\psi'(t) -t]/t^{p+1}$ is strictly increasing,
	\end{itemize}
	then the sequence $\{t_{k+1}/t_k^{p+1}\}$ is strictly decreasing and there holds
	\begin{equation}\label{eq:rateconvergence}
	\|x_{k+1}- \bar x\|\leq \frac{t_{k+1}}{t_k^{p+1}}\|x_k -\bar x\|^{p+1},   \qquad  \qquad k=0,1,\ldots\,.
	\end{equation}
If, additionally,  $\psi(\rho)/(\rho \psi'(\rho))-1=1$ and $\rho < \kappa$, then  $r_{\bar x}=\rho$ is the biggest  convergence radius.
\end{theorem}
\begin{remark}
 The first equation in \eqref{eq:linearconv} means that $\{x^k\}$ converges  superlinearly to $\bar{x}$. Note that always  $\psi$ has derivative $\psi'$ convex,  condition  {\bf  h3} holds with $p=1$.  In this case,  there holds
 $$
 t_{k+1}/t_k^{2}\leq [\psi''(t_0)]/ [2 |\psi'(t_0)|], \qquad  \qquad k=0,1,\ldots\,
 $$
 and  $\{x^k\}$  converges  quadratically. Indeed,  convexity of $\psi'$ is  necessary  to obtain the quadratic  convergence; see Example 2 in  \cite{Ferreira2011}. Moreover, as $\{t_{k+1}/ t_k^{p+1}\}$ is strictly decreasing we have  $t_{k+1}/ t_k^{p+1}\leq t_1/t_0^{p+1}$, for  $k=0, 1, \ldots.$  Thus, \eqref{eq:rateconvergence} implies $\|x_{k+1}-\bar x\| \leq \left[t_{1}/t_0^{p+1}\right]\|x_k-\bar{x}\|^{p+1},$
for $ k=0,1,\ldots\, $. Consequently, if $p=0$ then $\|x_k-\bar{x}\|\leq t_0[t_{1}/t_0]^k$ for $ k=0,1,\ldots\,$ and if $0<p\leq 1$ then
$$
\|x_k- \bar x\|\leq t_0(t_{1}/t_0)^{[(p+1)^k -1]/p}, \qquad k=0,1,\ldots\,.
$$
\end{remark}
\begin{remark}
Note that throughout  the proof of the above theorem,  if we assume that  $F\equiv \{0\}$ then  the constant   $r_{\bar{x}}=\nu$. In this case, Theorem~\eqref{th:nt} merges  into   Theorem 2 of \cite{Ferreira2011}.
\end{remark}
 From now on, we assume that the hypotheses of Theorem \ref{th:nt} hold, with the exception of {\bf h3}, which will be considered to hold only when explicitly stated.
\subsection{Preliminary results} \label{sec:PMF}
In this section, our first goal is to prove all statements in Theorem~\ref{th:nt} concerning the sequence $\{t_k\}$  associated to the majorant function $\psi$ defined in \eqref{eq:DNS}.  Moreover, we obtain  some  relationships between the majorant function $\psi$ and the set-valued mapping  $f+F$, which will play an important role throughout the paper. Furthermore, the results in Theorem~\ref{th:nt} related to the uniqueness of the solution and the optimal convergence radius will be proved. We begin with some observations on the majorant function.    

As proven in  Proposition~2.5 of \cite{Ferreira2009},  the constants $\kappa$,  $\nu$ and $\sigma$, defined in Definition~\ref{def.majorcondition} and Theorem~\ref{th:nt},  are all positives and $t-\psi(t)/\psi'(t)<0,$ for all $t\in (0,\,\nu).$  According to {\bf h2} and  definition of $\nu$, we have  $\psi'(t)< 0$,  for all
$t\in[0, \,\nu)$.  Therefore, the Newton iteration map for  $\psi$ is well defined in
$[0,\, \nu)$, namely, $ n_{\psi}:[0,\, \nu)\to (-\infty, \, 0]$  is defined by
\begin{equation} \label{eq:def.nf}
n_{\psi}(t):=t-\psi(t)/\psi'(t), \qquad  t\in [0,\, \nu).
\end{equation}
The next proposition was proved in Proposition~4 of \cite{Ferreira2011}.
\begin{proposition}  \label{pr:incr2}
$\lim_{t\to 0} |n_{\psi}(t)|/t =0$  and  the constant $ \rho $ is positive. As a consequence,  $|n_{\psi}(t)|<t$ for all $ t\in (0, \, \rho)$.
\end{proposition}
Using \eqref{eq:def.nf}, it is easy to see that  the sequence $\{t_k \}$ is equivalently defined as
\begin{equation} \label{eq:tknk}
 t_0=\|\bar{x}-x_0\|, \qquad t_{k+1}=|n_{\psi}(t_k)|, \qquad k=0,1,\ldots\, .
\end{equation}
Next result,  which is a consequence of above proposition,   contains  the  main convergence  properties of the above sequence and its prove can be found in  Corollary~5 of \cite{Ferreira2011}.
\begin{corollary} \label{cr:kanttk}
The sequence $\{t_k\}$ is well defined, is strictly decreasing and is contained in $(0, \rho)$. Moreover, $\{t_k\}$ converges to $0$ with superlinear rate, i.e.,$$\lim_{k\to \infty} t_{k+1}/{t_k}=0.$$ If additionally {\bf h3} holds,  then the sequence $t_{k+1}/t_k^{p+1}$  is strictly decreasing.
\end{corollary}
In the sequel we study the  {\it linearization error   of the function $f$} at a point in $\Omega$,  namely,
\begin{equation}\label{eq:def.er}
  E_f(x,y):= f(y)-\left[ f(x)+f'(x)(y-x)\right],\qquad y,\, x\in \Omega.
\end{equation}
We show that this error is bounded  by the linearization error  of the majorant function $\psi$, i.e.,
$$
        e_{\psi}(t,u):= \psi(u)-\left[ \psi(t)+\psi'(t)(u-t)\right],\qquad t,\,u \in [0,R),
$$
and as consequence, we  prove that the partial linearization of $f+F$ has  a single-valued inverse,  which is Lipschitz  in a  neighborhood of $\bar x$. 
\begin{lemma}  \label{pr:taylor}
There holds $\lambda \|E_f(x, \bar{x})\|\leq e_{\psi}(\|x-\bar{x}\|, 0),$ for all $x\in B(\bar{x}, \kappa)$.
\end{lemma}
\begin{proof}
 Since   $\bar{x}+(1-u)(x-\bar{x})\in B(\bar{x}, \kappa)$, for all $0\leq u\leq 1$ and   $f$ is  continuously differentiable in $\Omega$, thus the definition of $E_f$ and some simple manipulations yield
$$
\lambda\|E_f(x,\bar{x})\|\leq \int_0 ^1 \lambda  \left \| f'(x)-f'(\bar{x}+(1-u)(x-\bar{x}))]\right\|\,\left\|\bar{x}-x\right\| \; du.
$$
Combining last inequality with \eqref{Hyp:MH} and then performing the integral obtained and using the definition of $e_{\psi}$, the statement follows.
\end{proof}
The next result states  that,  if a generalized equation  \eqref{eq:ipi} is strongly regular at $\bar{x}$ for $0$ with modulus $\lambda >0$ and  \eqref{Hyp:MH} holds, then  there exists a neighborhood of $\bar{x}$ such that,  for all $x$  in this neighborhood,   \eqref{eq:ipi} is also  strongly regular at $x$ for $0$ with   modulus $ \lambda/(|\psi'(\|x-\bar x\|)|)$. The  result is a consequence of Corollary~\ref{cor:ban} and its statement  is:
\begin{lemma} \label{le:wdns}
There exists a constant $r_{\bar{x}} \leq r$ such that,    the mapping $$ x\mapsto L_f(x, 0)^{-1}\cap B(\bar{x}, r_{\bar{x}})$$ is   single-valued in  $ B(\bar{x}, r_{\bar{x}})$ and there holds
$$
\left\|{\bar{x}}- L_f(x,0)^{-1}\cap B(\bar{x}, r_{\bar{x}}) \right\| \leq  \frac{\lambda}{|\psi'(\|x-\bar x\|)|} \|E_f(x,\bar{x})\|, \qquad \forall ~ x\in B(\bar{x}, r_{\bar{x}}) .
$$
\end{lemma}
\begin{proof}
Take  $x\in B(\bar{x}, r)$.   Since $ r<\nu$ we have  $\| x-\bar{x}\|<\nu$. Thus,   $\psi'(\|x-\bar x\|)<0$ which, together  \eqref{Hyp:MH} and {\bf h1},  imply that
  \begin{equation}\label{eq:majcond}
     \lambda\|f'(x)-f'(\bar x)\| \leq \psi'(\|x-\bar x\|)-\psi'(0)<-\psi'(0)=1, \qquad  \forall ~x\in B(\bar{x}, r).
  \end{equation}
Due to $f+F$  be  strongly regular  at $\bar{x}$ for $0$ with modulus  $\lambda >0$, we can  apply  Corollary \ref{cor:ban} to obtain $r_{\bar{x}}>0$ and $r_{0}>0$ such that, for any $x\in B(\bar{x}, r_{\bar{x}})$,   the mapping $ z\mapsto L_f(x,z)^{-1}\cap B(\bar{x}, r_{\bar{x}})$ is   single-valued  from $B(0, r_{0})$ to  $B(\bar{x}, r_{\bar{x}})$. In particular, we conclude that the  mapping $ x\mapsto L_f(x, 0)^{-1}\cap B(\bar{x}, r_{\bar{x}})$ is   single-valued in  $ B(\bar{x}, r_{\bar{x}})$.  Moreover,  Corollary \ref{cor:ban} implies that $\forall ~ u, v \in  B(0, r_{0})$
$$
\left\|L_f(x,u)^{-1}\cap B(\bar{x}, r_{\bar{x}})- L_f(x,v)^{-1}\cap B(\bar{x}, r_{\bar{x}}) \right\| \leq \frac{\lambda \|u-v\|}{1-\lambda\|f'(x)-f'(\bar{x})\|} .
$$
If necessary, we  shrink  $r_{\bar{x}}$ such that $r_{\bar{x}} \leq r$, in order to   combine the last  inequality with  the first inequality in  \eqref{eq:majcond} and {\bf h1},  to  conclude  that,  for all $x\in B(\bar{x}, r_{\bar{x}})$ there holds
\begin{equation} \label{eq:gcl}
\left\|L_f(x,u)^{-1}\cap B(\bar{x}, r_{\bar{x}})- L_f(x,v)^{-1}\cap B(\bar{x}, r_{\bar{x}}) \right\| \leq  \frac{\lambda \|u-v\|}{|\psi'(\|x-\bar x\|)|} ,
\end{equation}
for each $ u, v \in  B(0, r_{0}).$ On the other hand, due to   $f$ be   continuous with    $f'$ continuous in $\Omega$,  we have $\lim_{x\to \bar{x}} E_f(x,\bar{x}) =0$. Thus,  we can shrink $r_{\bar{x}}$,  if necessary,  such that
$$
E_f(x,\bar{x}) \in B(0,r_0), \qquad \forall ~ x\in B(\bar{x}, r_{\bar{x}}).
$$
Let $x\in B(\bar{x}, r_{\bar{x}})$.  Note that, after    some algebraic manipulations  we obtain
\begin{align*}
0\in f(\bar{x}) + F(\bar{x})&= f(x) + f'(x)(\bar{x}-x)-f(x)-f'(x)(\bar{x}-x) +f(\bar{x}) + F(\bar{x})\\
                                       &= f(x) + f'(x)(\bar{x}-x) + E_f(x,\bar{x}) + F(\bar{x}).
\end{align*}
Thus, $-E_f(x,\bar{x})\in  L_f(x,\bar{x})= f(x) + f'(x)(\bar{x}-x) + F(\bar{x})$.  Since  $E_f(x,\bar{x}) \in B(0,r_0)$ and   the mapping $ z\mapsto L_f(x,z)^{-1}\cap B(\bar{x}, r_{\bar{x}})$ is   single-valued  from $B(0, r_{0})$ to  $B(\bar{x}, r_{\bar{x}})$  we conclude that
$$
\bar{x}=L_f(x, -E_f(x,\bar{x}))^{-1}\cap B(\bar{x}, r_{\bar{x}}).
$$
Therefore, substituting $u= -E_f(x,\bar{x})$ and $v=0$ into \eqref{eq:gcl} the desired inequality follows.
\end{proof}
Let   $r_{\bar{x}}>0$ the constant  given by Lemma~\ref{le:wdns}.  Lemma~\ref{le:wdns} guarantees, in particular,  that  the mapping $ x\mapsto L_f(x,0)^{-1}\cap B(\bar{x}, r_{\bar{x}})$ is   single-valued  in   $ B(\bar{x}, r_{\bar{x}})$  and consequently, the Newton iteration mapping  is well-defined.  Let us call $N_{f+F}$, the Newton
iteration mapping  for $f+F$ in that region, namely, $N_{f+F}:B(\bar{x}, r_{\bar{x}}) \to \banacha$ is defined by
\begin{equation} \label{NF}
N_{f+F}(x):= L_f(x,0)^{-1}\cap B(\bar{x}, r_{\bar{x}}), \qquad \forall ~x\in B(\bar{x}, r_{\bar{x}}).
\end{equation}
Using \eqref{eq:invplm}, the definition of Newton iteration mapping in \eqref{NF} is equivalent to
\begin{equation} \label{eq:NFef}
0\in f(x)+f'(x)(N_{f+F}(x)-x)+F(N_{f+F}(x)),\quad  N_{f+F}(x)\in B(\bar{x}, r_{\bar{x}}),   
\end{equation}
for each $x\in  B(\bar{x}, r_{\bar{x}})$. Therefore, since Lemma~\ref{le:wdns} guarantees that $N_{f+F}(x)$ is single-valued at $B(\bar{x}, r_{\bar{x}})$, see \eqref{NF}, we can apply a \emph{single} Newton iteration for any $x\in B(\bar{x}, r_{\bar{x}})$ to obtain $N_{f+F}(x)$ which may not belong
to $B(\bar{x},  r_{\bar{x}})$, or even may not belong to the domain of $f$. Thus, this allow us  to guarantee the  well-definedness of  only one iteration of Newton's method.  In particular,  the next result shows that for any $x\in   B(\bar{x}, r_{\bar{x}})$, the  Newton  iterations, see \eqref{eq:NFef},  may be repeated indefinitely.
\begin{lemma} \label{le:cl}
If   $\|x-\bar{x}\|\leq t< r_{\bar{x}}$  then $\|N_{f+F}(x)-\bar{x}\|\leq  |n_{\psi}(\|x-\bar{x}\|)|.$  As a consequence,
 $N_{f+F}(B(\bar{x}, r_{\bar{x}}))\subset B(\bar{x}, r_{\bar{x}})$.  Moreover, if {\bf h3} holds and $x\neq \bar{x}$ then $ \|N_{f+F}(x)-\bar{x}\|\leq [ |n_{\psi}(t)|/t^{p+1}] \|x-\bar{x}\|^{p+1}.$
\end{lemma}
\begin{proof}
Since $0\in f(\bar{x}) + F(\bar{x})$  we have $ \bar{x}=N_{f+F}(\bar{x})$. Thus,  the inequality of the lemma is trivial for $x=\bar{x}$. Now, assume that  $0<\|x-\bar{x}\|\leq t$. Hence, Lemma~\ref{le:wdns} implies that  the mapping $ x\mapsto L_f(x,0)^{-1}\cap B(\bar{x}, r_{\bar{x}})$ is   single-valued  in   $ B(\bar{x}, r_{\bar{x}})$ and Lipschitz continuous with modulus $\lambda /|\psi'(\|x-\bar{x}\|)|$. Using   \eqref{NF} and Lemma \ref{le:wdns}, it is easy to conclude that
\begin{equation*} \label{eq:il10}
 \|\bar{x}-N_{f+F}(x)\| \leq \frac{\lambda}{|\psi'(\|x-\bar{x}\|)|} \|E_{f}(x,\bar{x})\|.
\end{equation*}
Now, applying Lemma~\ref{pr:taylor} we obtain
$$
\|\bar{x}-N_{f+F}(x)\| \leq \frac{e_{\psi} (\|x-\bar{x}\|,0)}{|\psi'(\|x-\bar{x}\|)|}.
$$
On the other hand, taking into account  that $\psi(0)=0$,   the definitions of $e_\psi$ and  $n_\psi$  imply that
$$
\frac{e_{\psi} (\|x-\bar{x}\|,0)}{|\psi'(\|x-\bar{x}\|)|} =\frac{\psi(\|x-\bar{x}\|)}{\psi'(\|x-\bar{x}\|)} -\|x-\bar{x}\|= |n_{\psi}(\|x-\bar{x}\|)|.
$$
Hence, the first statement follows by combining the above two expressions.
For proving the inclusion  of the lemma, take $x\in B(\bar{x}, r_{\bar{x}})$.  Since $\|x-\bar{x}\|< r_{\bar{x}} $, $r_{\bar{x}}\leq\rho$ and $\|N_{f+F}(x)-\bar{x}\|\leq  |n_{\psi}(\|x-\bar{x}\|)|$, thus using the  second part of Proposition~\ref{pr:incr2} we conclude that
$
\|N_{f+F}(x)-\bar{x}\|< \|x-\bar{x}\|
$
which prove the  inclusion.

In the following, we prove  last inequality. Due $0\in f(\bar x) + F(\bar x)$, the inequality is trivial for $x=\bar x$. If $0 < \|x-\bar{x}\| \leq t$ then assumption {\bf h3} and \eqref{eq:def.nf} yields
$$
\frac{|n_{\psi}(\|x-\bar{x}\|)|}{\|x-\bar{x}\|^{p+1}} < \frac{|n_{\psi}(t)|}{t^{p+1}}.
$$
Therefore, using the first part of Lemma~\ref{le:cl} the inequality  follows.
\end{proof}

In the next result we  obtain the uniqueness of the solution  in the neighborhood $B[\bar x, \sigma].$
\begin{lemma}\label{l:uniq}
There exists a constant  $\bar{\sigma}\leq\min \{r_{\bar{x}}, \sigma\}$ such that $\bar{x}$ is the unique solution of \eqref{eq:ipi}  in $B[\bar{x}, \bar{\sigma}]$.
\end{lemma}
\begin{proof}
Let $r_{\bar{x}}>0$  the constant  given by Lemma~\ref{le:wdns}. Thus,   Corollary~\ref{cor:ban} implies that there exists $r_{0}>0$ such that, for any $x\in B(\bar{x}, r_{\bar{x}})$,   the mapping $ z\mapsto L_f(x,z)^{-1}\cap B(\bar{x}, r_{\bar{x}})$ is   single-valued  from $B(0, r_{0})$ to  $B(\bar{x}, r_{\bar{x}})$ and there holds
$$
\left\|L_f(x,u)^{-1}\cap B(\bar{x}, r_{\bar{x}})- L_f(x,v)^{-1}\cap B(\bar{x}, r_{\bar{x}}) \right\| \leq \frac{\lambda \|u-v\|}{1-\lambda\|f'(x)-f'(\bar{x})\|},
$$
for each $u, v \in  B(0, r_{0})$. Now, due to   $f$ be   continuous,  we have $$\lim_{x\to \bar{x}} E_f(\bar{x}, x) =0.$$ Thus,  we can take $\bar{\sigma}\leq\min \{r_{\bar{x}}, \sigma\}$,  such that
$$
E_f(\bar{x}, y) \in B(0,r_0).
$$
Let $y\in B(\bar{x}, \bar{\sigma})$ and  assume that  $0\in f(y)+F(y).$    Then,   some  manipulations  yield
\begin{align*}
0\in f(y) + F(y)&= f(y)-f(\bar{x}) - f'(\bar{x})(y-\bar{x})+f(\bar{x})+f'(\bar{x})(y-\bar{x})  + F(y)\\
                                       &= E_f(\bar{x},y) + L_f(\bar{x},y),
\end{align*}
which implies that  $-E_f(\bar{x},y)\in  L_f(\bar{x},y)$. Since the mapping $ z\mapsto L_f(\bar{x},z)^{-1}\cap B(\bar{x}, r_{\bar{x}})$ is   single-valued  from $B(0, r_{0})$ to  $B(\bar{x}, r_{\bar{x}})$,  we have
$$
y=L_f(\bar{x}, -E_f(\bar{x},y))^{-1}\cap B(\bar{x}, r_{\bar{x}}), \qquad \bar{x}=L_f(\bar{x}, 0)^{-1}\cap B(\bar{x}, r_{\bar{x}}).
$$
Thus,  substituting into above inequality  $x=\bar{x}$, $u=0$ and $v=-E_f(\bar{x},y)$, we conclude that
$$
\|\bar{x}-y\|=\|L_f(\bar{x}, 0)^{-1}\cap B(\bar{x}, r_{\bar{x}})-L_f(\bar{x}, -E_f(\bar{x},y))^{-1}\cap B(\bar{x}, r_{\bar{x}})\| \leq \lambda \|E_f(\bar{x},y)\|.
$$
Using definition on \eqref{eq:def.er} and \eqref{Hyp:MH} with $x=\bar{x}+u(y-\bar{x})$ and $\tau=0$, last inequality implies
\begin{eqnarray*}
\|\bar{x}-y\| &\leq& \lambda \|f(y)-f(\bar{x}) -f'(\bar{x})(y-\bar{x})\| \\
              &\leq&  \int_{0}^{1} \lambda \left\| f'(\bar{x}+u(y-\bar{x}))-f'(\bar{x})\right\|  \|y-\bar{x}\|du\\
              &\leq&   \int_{0}^{1} [\psi'(u\|y-\bar{x})\|)-\psi'(0)] \|y-\bar{x}\|du.
\end{eqnarray*}
Performing the integral of the right hand side of the above inequality we have $0\leq \psi(\|y-\bar{x}\|)$, which implies that $ \psi(\|y-\bar{x}\|)=0$ due to    $\psi (t)<0$ for $t\in (0, \sigma)$ and $\|y-\bar{x}\|\leq \sigma$.  Since  $0\leq \|y-\bar{x}\|\leq {\sigma}$ and $0$ is the unique zero of $\psi$ in $[0, \sigma],$ we conclude that $\|y-\bar{x}\|=0$, i.e., $y=\bar{x}$ and $\bar{x}$ is the unique solution of  \eqref{eq:ipi} in $B[\bar{x},\bar{\sigma}].$
\end{proof}
The next result  gives  the biggest convergence radius,  its  proof is similar  to the proof of Lemma~{2.15} of \cite{Ferreira2009}.
\begin{lemma}\label{pr:best}
If  $\psi(\rho)/(\rho \psi'(\rho))-1=1$ and $\rho < \kappa$, then  $r_{\bar x}=\rho$ is the biggest convergence radius.
\end{lemma}
\subsection{Proof of {\bf Theorem \ref{th:nt}}} \label{sec:proof}
First, note that the inclusion  in \eqref{eq:DNS} together  \eqref{NF} and \eqref{eq:NFef} imply that   the sequence $\{x_k\}$  satisfies
\begin{equation} \label{NFS}
x_{k+1}=N_{f+F}(x_k),\qquad k=0,1,\ldots \,,
\end{equation}
which is indeed an equivalent definition of this sequence.
\begin{proof}
All statements  involving  $\{t_k\}$ were proved in Corollary \ref{cr:kanttk}.  Since  Lemma~\ref{le:wdns} and     \eqref{NF} implies that there exist constants $r_{\bar{x}}>0$ and  $r_{0}>0$ such that $r_{\bar{x}} \leq r$ and,  for any $x\in B(\bar{x}, r_{\bar{x}})$,  the mapping $ N_{f+F}$ is single-valued in $B(\bar{x},r_{\bar{x}})$. Thus, taking into account that    $x_0\in B(\bar{x},r_{\bar{x}})$, we conclude by combining  \eqref{NFS} and   inclusion  $N_{f+F}(B(\bar{x}, r_{\bar{x}})) \subset B(\bar{x}, r_{\bar{x}})$ in  Lemma~\ref{le:cl} that   $\{x_k\}$  is well defined and remains  in $B(\bar{x}, r_{\bar{x}})$.  Now, we are going to prove that $\{x_k \}$ converges towards $\bar{x}$. Without lose generality we assume that the sequence $\{x_k \}$ is infinity. Since $0< \|x_k-\bar{x}\|<r_{\bar{x}}\leq \rho,$ for $k=0,1,\ldots,$ we obtain from \eqref{NFS}, Lemma~\ref{le:cl} and second part of Proposition~\ref{pr:incr2} that
\begin{equation}\label{eq:conv}
\|x_{k+1}-\bar{x}\| \leq |n_{\psi}(\|x_k-\bar{x}\|)|<\|x_k-\bar{x}\|, \qquad k=0,1,\ldots .
\end{equation}
Thus, $\{\|x_k-\bar{x}\|\}$ is strictly decreasing and convergent. Let $\bar{\alpha}=\lim_{k \to \infty} \|x_k-\bar{x}\|.$ Because $\{\|x_k-\bar{x}\|\}$ rest in $(0,\rho)$ and it is  strictly decreasing we have $0\leq \bar{\alpha}<\rho.$ Then, by continuity of $n_{\psi}$ and \eqref{eq:conv} imply $0\leq \bar{\alpha}=|n_{\psi}(\bar{\alpha})|,$ and from second part of Proposition~\ref{pr:incr2} we have $\bar{\alpha}=0.$ Therefore, the convergence of $\{x_k\}$ to $\bar{x}$  is proved. Now we are going show that $\bar{x}$ is a solution of the generalized equation $f(x) + F(x) \ni 0.$ From  inclusion in \eqref{eq:DNS} we conclude
$$
\left(x_{k+1}, -f(x_k)-f'(x_k)(x_{k+1}-x_k)\right) \in \mbox{gph} ~F, \qquad k=0,1,\ldots .
$$
By assumption the set-valued mapping $F$ has closed graph and   $f$  is  continuous with      $f'$ continuous, thus last inclusion   implies that
$$
\lim_{k\to \infty }\left((x_{k+1}, -f(x_k)-f'(x_k)(x_{k+1}-x_k)\right)=(\bar{x}, -f(\bar{x}))\in  \mbox{gph} ~F,
$$
 which implies $f(\bar{x}) + F(\bar{x}) \ni 0.$ Now, we are going to show the  first inequality in  \eqref{eq:linearconv}. Note that \eqref{eq:conv} implies
$$
\frac{\|x_{k+1}-\bar{x}\|}{\|x_k-\bar{x}\|} \leq \frac{|n_{\psi}(\|x_k-\bar{x}\|)|}{\|x_k-\bar{x}\|}, \qquad k=0,1,\ldots.
$$
Since $\lim_{k\to \infty} \|x_k-\bar{x}\|=0$, the desired equality follows from the first statement in Proposition~\ref{pr:incr2}.
To prove \eqref{eq:rateconvergence}, firstly we will show by induction that the sequences   $\{x_k\}$ and $\{t_k\}$ defined  in \eqref{eq:DNS} satisfy
\begin{equation}\label{eq.indu}
\|x_k-\bar x\|\leq t_k, \qquad k=0,1,\ldots.
\end{equation}
Since $t_0=\|x_0-\bar x\|$, the above inequality holds to $k=0$. Now, we assume that $\|x_k-\bar x\| \leq t_k$ holds. Using \eqref{NFS}, second part of Lemma~\ref{le:cl}, the induction assumption and \eqref{eq:tknk} we have
$$
\|x_{k+1}-\bar{x}\|=\|N_{f+F}(x_{k})-\bar{x}\|\leq \frac{|n_{\psi}(t_k)|}{t_k^{p+1}}  \|x_k-\bar{x}\|^{p+1} = \frac{t_{k+1}}{t_k^{p+1}}  \|x_k-\bar{x}\|^{p+1}\leq t_{k+1},
$$
and the proof by induction is complete. Thus, the inequality  \eqref{eq:rateconvergence} follows by combination of  \eqref{eq.indu} and second part of Lemma~\ref{le:cl}.   Finally, the uniqueness follows from  Lemma~\ref{l:uniq} and   the last statement in the theorem follows from   Lemma~\ref{pr:best}.
\end{proof}
\section{Particular cases} \label{apl}
In this section, some  special cases of  Theorem \ref{th:nt} will be considered. For instance, if $F\equiv \{0\}$  and    $f'$ satisfies a H\"older-type condition, a particular instance of Theorem~\ref{th:nt}, which retrieves  the classical convergence theorem on Newton's method under the Lipschitz condition will be obtained; see \cite{Rall1974, Traub1979}.  We also obtain  Theorem~1 of N.~H.~Josephy  in \cite{josephy1979} and, up to some minor adjustments, Theorem~1 of A.~L.~Dontchev \cite{Dontchev1996}. To complete this section, a version of Smale's theorem on Newton's method for analytical functions is proved in Theorem~\ref{theo:Smale}.
\subsection{Under H\"older-type condition}
The next result, which is a consequence of our main result Theorem~\ref{th:nt}, is a version of  classical convergence theorem for Newton's method under H\"older-type condition  for solving generalized equations of the type \eqref{eq:ipi}. Classical  versions  for  $F\equiv \{0\}$   have  appeared   in  \cite{Huang2004, Li2008a, Rall1974, Traub1979}.
\begin{theorem} \label{th:cc}
Let $\banacha$, $\banachb$ be Banach spaces, $\Omega\subseteq \banacha$ an open set and
  $f:{\Omega}\to \banachb$ be continuous with Fr\'echet derivative $f'$ continuous in $\Omega$, $F:\banacha \rightrightarrows  \banachb$ be a set-valued mapping with closed graph and $\bar{x}\in \Omega$.  Suppose that $f+F$ is strongly regular at $\bar{x}$ for $0$ with modulus $\lambda >0$ and  there exist  constants $K>0$ and $0<p\leq 1$ such that
\begin{equation}\label{eq:hc}
\lambda \left\|f'(x)-f'(\bar{x}+\tau(x-\bar{x}))\right\| \leq   (K-\tau^p)\|x-\bar{x}\|^p, \;  x\in B(\bar{x}, \kappa), \;  \tau \in [0,1].
 \end{equation}
Let $r:=\min \left\{\kappa, \,[(p+1)/((2p+1)K]^{1/p}\right\}$, where $\kappa:=\sup\{t>0: B(\bar{x}, t)\subset \Omega\}$. Then, there exists    a convergence radius  $r_{\bar{x}}>0$ with $r_{\bar{x}}\leq r$ such that  the sequences  with starting point $x_0\in B(\bar{x}, r_{\bar{x}})/\{\bar{x}\}$ and $t_0=\|\bar{x}-x_0\|$, respectively,
\begin{equation} \label{eq:Kant}
 f(x_k)+f'(x_k)(x_{k+1}-x_k)+F(x_{k+1})\ni 0, \qquad t_{k+1} = \frac{Kpt_k^{p+1}}{(p+1)[1-Kt_k^p]}, 
\end{equation}
$ k=0,1,\ldots\,,$ are well defined, $\{t_k\}$ is strictly decreasing, is contained in $(0, r)$ and converges to $0$, $\{x_k\}$ is contained in $B(\bar{x}, r_{\bar{x}})$ and  converges to the point $\bar{x}$ which is a unique solution of $f(x)+F(x)\ni 0$ in $B(\bar{x}, \bar{\sigma})$, where $\bar{\sigma}\leq\min \{r_{\bar{x}},  [(p+1)/K]^{1/p}\}$. Moreover, $\{t_{k+1}/ t_k^2\}$ is strictly decreasing, $t_{k+1}/t_k^2<1/[2/K-2\|\bar{x}-x_0\|]$  and
$$
 \|\bar{x}-x_{k+1}\| \leq \frac{Kp \|x_k-\bar{x}\|^{p+1}}{(p+1)[1-Kt_k^p]}  \leq \frac{Kp \|x_k-\bar{x}\|^{p+1}}{(p+1)[1-K\|x_0-\bar{x}\|^p]},  \quad k=0,1,\ldots\,.
$$
 If, additionally,  $[(p+1)/(2p+1)K]^{1/p}<\kappa$, then $r_{\bar x}=[(p+1)/(2p+1)K]^{1/p}$ is the biggest  convergence radius.
\end{theorem}
\begin{proof}
Using condition in \eqref{eq:hc},  we can immediately prove that  $f$, $\bar{x}$ and $\psi:[0, \kappa)\to \mathbb{R}$, defined by
$
\psi(t)=Kt^{p+1}/(p+1)-t,
$
satisfy the inequality \eqref{Hyp:MH} and the conditions  {\bf h1},   {\bf h2} and   {\bf h3}  in Theorem~\ref{th:nt}. In this case, it is easy to see that  $\rho$ and $\nu$, as defined in Theorem \ref{th:nt}, satisfy
$
\rho=[(p+1)/(2p+1)K]^{1/p} \leq \nu=[1/K]^{1/p}
$
and, as a consequence,  $r:=\min \{\kappa,\; [(p+1)/((2p+1)K]^{1/p}\}$. Moreover,   $\psi(\rho)/(\rho \psi'(\rho))-1=1$, $\psi(0)=\psi([(p+1)/K]^{1/p})=0$ and $\psi(t)<0$ for all $t\in (0,\, [(p+1)/K]^{1/p})$. Also,  the sequence $\{t_k\}$  in  Theorem \ref{th:nt} is given by  \eqref{eq:Kant}  and satisfies
$$
t_{k+1}/t_k^2= \frac{Kp}{(p+1)[1-Kt_k^p]}< \frac{Kp}{(p+1)[1-K\|x_0-\bar{x}\|^p]},  \qquad k=0,1,\ldots.
$$
Therefore, the result follows  by invoking Theorem~\ref{th:nt}.
\end{proof}
\begin{remark}
Theorem~\ref{th:cc}  contain,  as particular instance, some results on Newton's method; as we can see in, Rall \cite{Rall1974} and   Traub and  Wozniakowski \cite{Traub1979}.
\end{remark}
We are going to  study the variational inequality problem, namely, the generalized equation associated to $F=N_C$  the normal cone  of  $C$ a nonempty,  closed and  convex subset of $\banachb$,
\begin{equation} \label{eq:ncge}
f(x)+N_{C}(x) \ni 0.
\end{equation}
The next result is a version of  classical convergence theorem for Newton's method under Lipschitz-type condition for the variational inequality  \eqref{eq:ncge}, it has been  prove by  N.~H.~Josephy  in  \cite{josephy1979}.
\begin{theorem} \label{th:jtnm}
Let $\banacha$, $\banachb$ be Banach spaces, $C$ a nonempty,  closed and  convex subset of $\banachb$,  $\Omega\subseteq \banacha$ an open set and  $f:{\Omega}\to \banachb$ be continuous with Fr\'echet derivative $f'$ continuous in $\Omega$  such that
$$
 \|f'(x)-f'(y)\| \leq L \|x-y\|,\qquad x,\, y\in \Omega,
$$
where $L>0$. Let   $\bar{x}\in \Omega$ and   suppose that $f+N_{C}$ is strongly regular at $\bar{x}$ for $0$ with modulus $\lambda >0.$   Let $r=\min \left\{\kappa, \,2/(3\lambda L) \right\}$, where  $\kappa=\sup\{t\in [0, R): B(\bar{x}, t)\subset \Omega\}$.  Then, there exists    a convergence radius $r_{\bar{x}}>0$ with $r_{\bar{x}}\leq r$ such that  the sequences  with starting point $x_0\in B(\bar{x}, r_{\bar{x}})/\{\bar{x}\}$ and $t_0=\|\bar{x}-x_0\|$, respectively,
$$
 f(x_k)+f'(x_k)(x_{k+1}-x_k)+N_C(x_{k+1})\ni 0, \qquad t_{k+1} =\left((\lambda L/2) \, t_k^2\right)/(1-\lambda Lt_k),
$$
$ k=0,1,\ldots\,,$  are well defined, $\{t_k\}$ is strictly decreasing, is contained in $(0, r)$ and converges to $0$, $\{x_k\}$ is contained in $B(\bar{x}, r_{\bar{x}})$ and  converges to the point $\bar{x}$ which is a unique solution of $f(x)+N_C(x)\ni 0$ in $B(\bar{x}, \bar{\sigma})$, where $0<\bar{\sigma}\leq \min \{r_{\bar{x}},  2/K\}$. Moreover, $\{t_{k+1}/ t_k^2\}$ is strictly decreasing, $t_{k+1}/t_k^2<1/[2/(\lambda L)-2\|\bar{x}-x_0\|]$  and
$$
 \|\bar{x}-x_{k+1}\| \leq \frac{\lambda L}{2} \, \frac{1}{1-\lambda Lt_k}\,\|x_k-\bar{x}\|^2  \leq \frac{\lambda L}{2}\, \frac{1}{1-\lambda L\|x_0-\bar{x}\|}\,\|x_k-\bar{x}\|^2,
$$
$k=0,1, \ldots .$ If, additionally,  $2/(3\lambda L)<\kappa$, then $r_{\bar x}=2/(3\lambda L)$ is the biggest  convergence radius.
\end{theorem}
\begin{proof}
The result  follows by applying Theorem~\ref{th:cc} with  $\tau=0$, $p=1$,  $K=\lambda L$ and $F=N_{C}$.
\end{proof}
A. L.   Dontchev \cite{Dontchev1996}  under Aubin continuity  of the mapping  $L_f(\bar{x}, \cdot )^{-1}:\banachb \rightrightarrows  \banacha$, defined by
\begin{equation} \label{eq:invplmnc}
L_f(\bar{x}, z )^{-1}:=\left\{y\in X ~:~ z\in f(\bar{x})+f'(\bar{x})(y-\bar{x})+N_C(y)\right\},
\end{equation}
 has shown that  the Newton's method for solving \eqref{eq:ncge} generates a sequence that  converges
$Q$-quadratically to a solution. Now,  our purpose  is to show that, if $\banacha= \banachb=\mathbb{R}^{n}$, $F=N_C$ and $C \subset \mathbb{R}^n$ is a nonempty, polyhedral and convex set, then  this particular instance of Theorem~1 of \cite{Dontchev1996} follows from Theorem~\ref{th:cc}. We begin with the formal definition of Aubin continuity;  for more details see  \cite{DontchevRockafellar96, DontchevRockafellar2009}. First we need the following definitions:  The {\it distance} from a point $v\in \mathbb{R}^{n}$ to a set $U\subset \mathbb{R}^{n}$  is $d(v,  U):=\inf \{\|v-u\|~: u\in U\}$ and  the {\it excess} from   the set $U$ to the set  $V$ is $e(V, U):=\sup\{ d(v,  U) ~:~v\in V\}$.
\begin{definition}
A mapping  $H: \mathbb{R}^m \rightrightarrows \mathbb{R}^n$  is said to be  Aubin continuous, at $\bar{y}\in \mathbb{R}^m$ for $\bar{x}\in \mathbb{R}^n$ with modulus $\alpha\geq 0$, if $\bar{x}\in H(\bar{y})$ and there exist  constants $a>0$ and $b>0$ such that
$$
e(H(y_1)\cap  B(\bar{x}, a),  H(y_2))\leq \alpha \|y_1-y_2\|, \qquad \forall ~ y_1, y_2 \in B(\bar{y}, b).
$$
\end{definition}
It has been shown in  \cite[Theorem 1]{DontchevRockafellar96}   that  if  $C\subset \mathbb{R}^n$ is  a polyhedral convex set, then   Aubin continuity of ${\cal{L}}_f(\bar{x}, \cdot )^{-1}$  is equivalent to   strong regularity of $f+N_{C}$.
Next we state, with  some  adjustment,  Theorem~1 of   \cite{Dontchev1996}; see also \cite{Dontchev1996_2}.
\begin{theorem}
Let  $C\subset \mathbb{R}^n$ be a polyhedral convex set,  $\Omega\subseteq \mathbb{R}^n$ an open set and  $f:{\Omega}\to \banachb$ be continuous with Fr\'echet derivative  $f'$ continuous in $\Omega$   such that
$$
 \|f'(x)-f'(y)\| \leq L \|x-y\|,\qquad  \forall ~ x,\, y\in \Omega,
$$
where $L>0$. Let   $\bar{x}\in \Omega$ and   suppose that   ${\cal{L}}_f(\bar{x}, \cdot )^{-1}: \mathbb{R}^m \rightrightarrows  \mathbb{R}^n$ defined in \eqref{eq:invplmnc}  is Aubin continuous  at $0\in \mathbb{R}^m$ for $\bar{x}\in \mathbb{R}^n$ with modulus $\alpha\geq 0$.  Let $r:=\min \left\{\kappa, \,2/(3\lambda L) \right\}$, where  $\kappa=\sup\{t\in [0, R): B(\bar{x}, t)\subset \Omega\}$.  Then, there exists    a convergence radius $r_{\bar{x}}>0$ with $r_{\bar{x}}\leq r$ such that  the sequences  with starting point $x_0\in B(\bar{x}, r_{\bar{x}})/\{\bar{x}\}$ and $t_0=\|\bar{x}-x_0\|$, respectively,
$$
 f(x_k)+f'(x_k)(x_{k+1}-x_k)+N_C(x_{k+1})\ni 0, \qquad t_{k+1} =\left((\lambda L/2) \, t_k^2\right)/(1-\lambda Lt_k),
$$
$ k=0,1,\ldots\,,$ are well defined, $\{t_k\}$ is strictly decreasing, is contained in $(0, r)$ and converges to $0$, $\{x_k\}$ is contained in $B(\bar{x}, r_{\bar{x}})$ and  converges to the point $\bar{x}$ which is a unique solution of $f(x)+N_C(x)\ni 0$  in $B(\bar{x}, \bar{\sigma})$, where $0<\bar{\sigma}\leq\min \{r_{\bar{x}},  2/K\}$.  Moreover, $\{t_{k+1}/ t_k^2\}$ is strictly decreasing, $t_{k+1}/t_k^2<1/[2/(\lambda L)-2\|\bar{x}-x_0\|]$  and
$$
 \|\bar{x}-x_{k+1}\| \leq \frac{\lambda L}{2} \, \frac{1}{1-\lambda Lt_k}\,\|x_k-\bar{x}\|^2  \leq \frac{\lambda L}{2}\, \frac{1}{1-\lambda L\|x_0-\bar{x}\|}\,\|x_k-\bar{x}\|^2,
$$
$ k=0,1, \ldots .$ If, additionally,  $2/(3\lambda L)<\kappa$, then $r_{\bar x}=2/(3\lambda L)$ is the biggest  convergence radius.
\end{theorem}
\begin{proof}
Since   $C\subset \mathbb{R}^n$  is a polyhedral convex set, \cite[Theorem 1]{DontchevRockafellar96}  implies that   Aubin continuity  of ${\cal{L}}_f(\bar{x}, \cdot )^{-1}$ at $0\in \mathbb{R}^m$ for $\bar{x}\in \mathbb{R}^n$ with modulus $\alpha\geq 0$,  is equivalent to   strong regularity of $f+N_{C}$ at $\bar{x}$ for $0$ with modulus $\alpha\geq 0$. Thus, the result  follows by  applying   Theorem~\ref{th:jtnm}.
\end{proof}
\subsection{Under Smale's-type condition}
In this section, we assume that $f$ is an analytic function and using the ideas listed in \cite{Alvarez2008}, we present a version of the classical convergence theorem for Newton's method for solving the generalized equation \eqref{eq:ipi}. The classical version  has  appeared   in   corollary of Proposition 3 pp.~195 of Smale \cite{Smale1986}, see also Proposition 1 pp.~157 and Remark 1 pp.~158 of  Blum, Cucker,  Shub, and Smale~\cite{BlumSmale1998}; see also \cite{Ferreira2009}. For stating  the  result we need of  the following definition.

 Let  $\Omega\subseteq \banacha$ and   $f:{\Omega}\to \banachb$ be an analytic function. The $n$-th derivative of $f$  at $x$ is a $n$-th multilinear map $f^n(x): \banacha \times \ldots \times \banacha \to \banacha$ and its norm is defined by
  \[
 \|f^n(x)\|=\sup\left\{\|f^n(x)(v_1, \dots, v_n) \| : v_1, \dots, v_n\in \banacha, ~ \,\|v_i\|\leq1, \, i=1, \ldots, n \right\}.
 \]
\begin{theorem} \label{theo:Smale}
Let $\banacha$, $\banachb$ be Banach spaces, $\Omega\subseteq \banacha$ an open set and
  $f:{\Omega}\to \banachb$ be an analytic function, $F:\banacha \rightrightarrows  \banachb$ be a set-valued mapping with closed graph and $\bar{x}\in \Omega$.  Suppose that $0\in f(\bar{x})+F(\bar{x})$ and $f+F$ is strongly regular at $\bar{x}$ for $0$ with modulus $\lambda >0.$  Suppose that
\begin{equation} \label{eq:SmaleCond}
   \gamma:= \sup _{ n > 1 }\left\| \frac  {\lambda f^{(n)}(\bar{x})}{n!}\right\|^{1/(n-1)}<+\infty.
\end{equation}
Let $r:=\min \{\kappa,\; (5-\sqrt{17})/(4\gamma)\}$ the convergence radius, where $$\kappa:=\sup\{t>0 ~:~ B(\bar{x}, t)\subset \Omega\}.$$  Then, there exists  a convergence radius  $r_{\bar{x}}>0$ with $r_{\bar{x}} \leq r$ such that  the sequences  with starting point $x_0\in B(\bar{x}, r_{\bar{x}})/\{\bar{x}\}$ and $t_0=\|\bar{x}-x_0\|$, respectively
$$
0\in f(x_k)+f'(x_k)(x_{k+1}-x_k)+F(x_{k+1}), \qquad  t_{k+1}=\gamma t_k^2/[2(1- \gamma t_k)^2-1],
$$
$ k=0,1,\ldots\,,$ are well defined, $\{t_k\}$ is strictly decreasing, contained in $(0, r)$ and converges to $0$, and $\{x_k\}$ is contained in $B(\bar{x},r_{\bar{x}})$ and  converges to the point $\bar{x}$ which is the unique solution of $f(x)+F(x)\ni 0$ in $B(\bar{x}, \bar{\sigma})$, where $0<\bar{\sigma}\leq\min \{r_{\bar{x}}, 1/(2\gamma)\}$.  Moreover, $\{t_{k+1}/ t_k^2\}$ is strictly decreasing, $t_{k+1}/t_k^2<\gamma/[2(1-\gamma \|x_0-\bar{x}\|)^2-1]$,   for  $  k=0,1,\ldots $ and
$$
    \|x_{k+1}-\bar{x}\| \leq  \frac{ \gamma}{2(1- \gamma t_k)^2-1}\;\|x_k-\bar{x}\|^2\leq \frac{ \gamma}{2(1- \gamma \|x_0-\bar{x}\|)^2-1}\;\|x_k-\bar{x}\|^2, 
$$
$k=0,1, \ldots.$ If, additionally,  $(5-\sqrt{17})/(4 \gamma)<\kappa$, then $r_{\bar x}=(5-\sqrt{17})/(4\gamma)$ is the biggest  convergence radius.
\end{theorem}
To proving Theorem~\ref{theo:Smale} we will need of the following results. The first one, gives us a condition  that is easier to check than condition
\eqref{Hyp:MH}, whenever  the functions under consideration  are twice continuously differentiable, and its proof is similar to Lemma~5.3 of  \cite{Alvarez2008}. The second one, gives a relationship  between  the second derivatives $f''$ and $ \psi''$,  which allow us to show that  $f$ and $\psi$ satisfy \eqref{Hyp:MH}, and its proof follows the same path of Lemma~22 of \cite{FGO11}.
\begin{lemma}\label{lem.cond1}
If $f: \Omega \subset \banacha \to \banachb$ is an analytic function, $\bar{x} \in \Omega$ and $B(\bar{x}, 1/ \gamma)\subset \Omega,$ where $\gamma$ is defined in \eqref{eq:SmaleCond}, then for all $x\in B(\bar{x}, 1/  \gamma),$ it holds that
$
\|f''(x)\|\leq 2  \gamma/(1-  \gamma\|x-\bar{x}\|)^3.
$
\end{lemma}

\begin{lemma} \label{lc}
Let $\banacha$, $\banachb$ be Banach spaces, $\Omega\subseteq \banacha$ be an open set,
  $f:{\Omega}\to \banachb$  be twice continuously differentiable. Let $\bar{x}\in \Omega$, $R>0$  and  $\kappa=\sup\{t\in [0, R): B(\bar{x}, t)\subset \Omega\}$. Let $\lambda >0$ and  \mbox{$\psi:[0,R)\to \mathbb {R}$} be twice continuously differentiable such that $ \lambda \|f''(x)\|\leqslant \psi''(\|x-\bar{x}\|),$
for all $x\in B(\bar{x}, \kappa)$, then $f$ and $\psi$ satisfy \eqref{Hyp:MH}.
\end{lemma}

\noindent
{\bf [Proof of Theorem \ref{theo:Smale}]}.
Let $\psi:[0, 1/ \gamma) \to \mathbb{R}$ be defined by $\psi(t)=t/(1- \gamma t)-2t$. It is easy to see that $\psi$ is  analytic and
$\psi(0)=0$,   $\psi'(t)=1/(1- \gamma t)^2-2$, $\psi'(0)=-1$, $\psi''(t)=2 \gamma/(1-\gamma t)^3$.
Thence,  $\psi$ satisfies {\bf h1}  and  {\bf h2}.  Now, we combine  Lemma~\ref{lc}  with  Lemma~\ref{lem.cond1}, to conclude  that $f$  and $\psi$ satisfy  \eqref{Hyp:MH}.  The constants,  $\nu$,  $\rho$ and $r$, as defined in Theorem \ref{th:nt}, satisfy
$$\rho= \frac{5-\sqrt{17}}{4 \gamma}< \nu=\frac{\sqrt{2}-1}{\sqrt{2} \gamma}< \frac{1}{\ \gamma},\qquad r=\min \left\{\kappa,\; \frac{5-\sqrt{17}}{4 \gamma}\right\}.
$$
Moreover, $\psi(\rho)/(\rho \psi'(\rho))-1=1$ and $\psi(0)=\psi(1/(2 \gamma))=0$ and $\psi(t)<0$ for $t\in (0,\, 1/(2\gamma))$. Also, $\{t_k\}$  satisfy
$$
t_{k+1}/t_k^2=\frac{ \gamma}{2(1- \gamma t_k)^2-1}<\frac{ \gamma}{2(1- \gamma \|x_0-\bar{x}\|)^2-1} ,  \qquad k=0,1,\ldots.
$$
Therefore, the result follows by applying the Theorem~\ref{th:nt}.
\qed
\section{Final remarks } \label{rf}
In this paper, under a general majorant condition, we present a new local  convergence analysis of the  Newton's method for solving the generalized equation \eqref{eq:ipi}. Our  approach is based in the  Banach Perturbation Lemma obtained by  S.~M.~Robinson in  \cite[Theorem~2.4]{Robinson1980} and used by Josephy in his Ph.D thesis \cite{josephy1979}. The majorant condition  allow to unify several  convergence results pertaining to  Newton's method. Besides,  following the same idea of this paper, as future works,  we propose to study the  inexact Newton's method for solving the problem \eqref{eq:ipi}
described by
$$
(f(x_k) +f'(x_k)(x_{k+1}-x_k)+F(x_{k+1}))\cap R_{k}(x_k, x_{k+1}) \neq \varnothing, \qquad \;\;k=0,1, \ldots , 
$$ 
where  $R_{k}: \banacha \times \banacha \rightrightarrows \banachb$ is a sequence of set-valued mappings with closed graphs, in order to support computational implementations of the method.  Furthermore, it would be interesting   to study the approach of  this paper  under a weak assumption than strong regularity, namely,  the regularity metric; see \cite{DontchevRockafellar2009}.


%
%

\end{document}